\documentclass[12pt]{amsart}
\usepackage[osf,sc]{mathpazo}
\usepackage{amssymb}

\usepackage{geometry}
\usepackage{bm}
\usepackage{graphicx}
\usepackage{hyperref}
\hypersetup{
	colorlinks=true, 
	linktoc=all,     
	linkcolor=blue,
	citecolor=red,
	filecolor=black,
	urlcolor=blue	
}
\usepackage{enumerate}
\usepackage[inline]{enumitem}
\makeatletter
\newcommand{\inlineitem}[1][]{%
	\ifnum\enit@type=\tw@
	{\descriptionlabel{#1}}
	\hspace{\labelsep}%
	\else
	\ifnum\enit@type=\z@
	\refstepcounter{\@listctr}\fi
	\quad\@itemlabel\hspace{\labelsep}%
	\fi} \makeatother
\newcommand{\ga}{\alpha}

\newcommand{\Gs}{\Sigma}

\newcommand{\Gom}{\Omega}

\newcommand{\subs}{\subset}

\newcommand{\bs}{\backslash}

\newcommand{\nin}{\notin}

\newcommand{\ti}{\tilde}

\newcommand{\mbb}{\mathbb}

\newcommand{\us}{\underset}

\newcommand{\lra}{\longrightarrow}

\newcommand{\Ra}{\Rightarrow}

\newcommand{\equ}[1]{%
	\begin{equation*}
		#1
	\end{equation*}
}
\newcommand{\equa}[1]{%
	\begin{equation*}
		\begin{aligned}
			#1
		\end{aligned}
	\end{equation*}
}

\newcommand{\rectangle}{{%
		\ooalign{$\sqsubset\mkern3mu$\cr$\mkern3mu\sqsupset$\cr}%
}}

\theoremstyle{plain}
\newtheorem{theorem}{Theorem}[section]

\newtheorem{lemma}[theorem]{Lemma}
\newtheorem{cor}[theorem]{Corollary}

\newtheorem{ques}[theorem]{Question}

\newtheorem{note}[theorem]{Note}
\newtheorem{notation}[theorem]{Notation}

\makeatletter
\def\namedlabel#1#2{\begingroup
	\def\@currentlabel{#2}%
	\label{#1}\endgroup
}
\makeatother

\newtheorem*{thmOmega}{\bf{Theorem} $\bm{\Gom}$}

\theoremstyle{definition}
\newtheorem{defn}[theorem]{Definition}

\theoremstyle{remark}

\numberwithin{equation}{section}

\begin{document}

\title[On the Factorization of Two Adjacent Numbers]{On the Factorization of Two Adjacent Numbers in Multiplicatively Closed Sets Generated by Two 
Elements}
\author[C.P. Anil Kumar]{C.P. Anil Kumar}
\address{No. 104, Bldg. 23, Lakshmi Paradise, 5th Main, 11th Cross, LN Puram, Bengaluru-560021}
\email{akcp1728@gmail.com}
\date{\sc \today}
\begin{abstract}
For two natural numbers $1<p_1<p_2$, with $\ga=\frac{\log(p_1)}{\log(p_2)}$ irrational, we describe, in main Theorem~\ref{theorem:FactRectangles} and in Note~\ref{note:AdjFact}, the factorization of two adjacent numbers in the multiplicatively closed subset $S=\{p_1^ip_2^j\mid i,j\in \mbb{N}\cup\{0\}\}$ using primary and secondary convergents of $\ga$. This suggests general Question~\ref{ques:AdjFact} for more than two generators which is still open. 
\end{abstract}
\maketitle
\section{\bf{Introduction}}
Continued fractions have been studied extensively in the theory of diophantine approximation. More so as a tool to prove results in this theory, for example, Hurwitz's theorem. A basic introduction to the theory of continued fractions is given in~\cite{MR0001185},~\cite{MR1083765}. A proof of Hurwitz's theorem is also mentioned in ~\cite{MR1083765} (Chapter $7$). The following is the question which this article concerns and we answer the question using the theory of continued fractions as a tool. 

\begin{ques}
\label{ques:AdjacentNumberFactorization}
Let $S=\{1=s_0<s_1<s_2<\cdots\}\subs \mbb{N}$ be a multiplicatively closed set generated by two natural numbers $p_1<p_2$ such that $\frac{\log(p_1)}{\log(p_2)}$ is irrational.
Let $s_k=p_1^ip_2^j\in S$ for some $k\geq 0$. What are the factorizations of the adjacent numbers $s_{k-1},s_{k+1}$ in terms of $p_1,p_2,i,j$?
\end{ques}
We answer this Question~\ref{ques:AdjacentNumberFactorization} using the simple continued fraction expansion of $\frac{\log(p_1)}{\log(p_2)}$ and its primary and secondary convergents.
In C.~P.~Anil~Kumar~\cite{MR3943663} (Theorem $3.2$), a question on the existence of arbitrary large gaps in $S$ has been answered affirmatively by another constructive technique. Also in Article~\cite{MR3943663} ( Lemma $5.4$ and Note $5.5$) a formula for the next number of $p_2^j$ in the multiplicatively closed set $S$ has been found. Here we prove this result as Corollary~\ref{cor:ArbLargeInt} of main Theorem~\ref{theorem:FactRectangles}.  

The following question for more than two generators is still open. 
\begin{ques}
\label{ques:AdjFact}
Let $T=\{1=t_0<t_1<t_2<\cdots\}\subs \mbb{N}$ be a finitely generated multiplicatively closed 
infinite set generated by positive integers $d_1,d_2,\cdots,d_n$ for $n>2$. Let $t_k=d_1^{i_1}d_2^{i_2}\cdots d_n^{i_n}$. How do we construct an explicit 
factorization of the elements $t_{k-1},t_{k+1}\in T$ in terms of the positive integers $d_j,i_j, 1\leq j\leq n$?
\end{ques}
Now we proceed to mention some notation, a required definition and state main Theorem~\ref{theorem:FactRectangles}.
\begin{notation}
Throughout this article, let $0<p_1<p_2$ be two positive integers such that $\frac {\log p_1}{\log p_2}$ is irrational. Let $S=\{p_1^ip_2^j\mid i,j\in \mbb{N}\cup \{0\}\}=\{s_0=1<s_1<s_2<\cdots\}$ be the multiplicatively closed set generated by $p_1,p_2$.	
\end{notation}
\begin{defn}[Non-negative Integer Co-ordinates of an Element]
	Any integer $n\in S$ can be uniquely expressible as $n=p_1^ip_2^j$. We associate the non-negative integer pair $(i,j)$ to $n$ which are the integer co-ordinates of the element $n$.
	So in particular there is a bijection (co-ordinatisation map) of $S$ with the grid $(\mbb{N}\cup \{0\})\times (\mbb{N}\cup \{0\})$.
\end{defn}
Main Theorem~\ref{theorem:FactRectangles} gives the decomposition of the factorization grids
\equa{&(\mbb{N}\cup \{0\})\times (\mbb{N}\cup \{0\})\text{ and }\\
	&\big((\mbb{N}\cup \{0\})\times (\mbb{N}\cup \{0\})\big)^{*}=\big(\mbb{N}\cup\{0\}\big)\times \big(\mbb{N}\cup\{0\}\big)\bs\{(0,0)\}} 
into rectangles which are related by local translations to describe the factorization of the number and its next number in an elegant manner. Now we state the main theorem.
\begin{thmOmega}
	\namedlabel{theorem:FactRectangles}{$\Gom$}
	Let $\{a_0=0,a_1,a_2,\cdots,\}$ be the continued fraction of $\frac {\log p_1}{\log p_2}$. Let $h_0=0,k_0=1$ and let $\{\frac{h_i}{k_i}\mid i\in \mbb{N},gcd(h_i,k_i)=1\}$ 
	be the sequence of primary convergents of $\frac {\log p_1}{\log p_2}$. 
	Consider the integer grid rectangles $\rectangle A^t_iB^t_iC^t_iD^t_i$ for $i\geq 1$ of dimensions $h_{2i}\times k_{2i}$ with co-ordinates given by 
	\equa{	A^t_i&=\big(k_{2i-1}+tk_{2i},0\big),\\
			B^t_i&=\big(k_{2i-1}+(t+1)k_{2i}-1,0\big),\\
			C^t_i&=\big(k_{2i-1}+(t+1)k_{2i}-1,h_{2i}-1\big),\\
			D^t_i&=\big(k_{2i-1}+tk_{2i}, h_{2i}-1\big), 0\leq t< a_{2i+1}.}
	The corresponding translated rectangles (translation applied to each point) denoted by $\rectangle \ti{A^t_i}\ti{B^t_i}\ti{C^t_i}\ti{D^t_i}$ of next numbers in the multiplicatively closed set are given by 
	\equa{	\ti{A^t_i}&=\big(0,h_{2i-1}+th_{2i}\big),\\
			\ti{B^t_i}&=\big(k_{2i}-1,h_{2i-1}+th_{2i}\big),\\
			\ti{C^t_i}&=\big(k_{2i}-1,h_{2i-1}+(t+1)h_{2i}-1\big),\\
			\ti{D^t_i}&=\big(0,h_{2i-1}+(t+1)h_{2i}-1\big),0\leq t< a_{2i+1}} 
	again of the same dimensions $h_{2i}\times k_{2i}$ with translation given by \equ{\rectangle\ti{A^t_i}\ti{B^t_i}\ti{C^t_i}\ti{D^t_i}=\rectangle A^t_iB^t_iC^t_iD^t_i+\big(-k_{2i-1}-tk_{2i},h_{2i-1}+th_{2i}\big).}
	
	Now consider the integer grid rectangles $\rectangle P^t_iQ^t_iR^t_iS^t_i$ for $i\geq 0$ of dimensions $h_{2i+1} \times k_{2i+1}$ with co-ordinates given by 
	\equa{	P^t_i&=\big(0,h_{2i}+th_{2i+1}\big),\\
			Q^t_i&=\big(k_{2i+1}-1,h_{2i}+th_{2i+1}\big),\\
			R^t_i&=\big(k_{2i+1}-1,h_{2i}+(t+1)h_{2i+1}-1\big),\\
			S^t_i&=\big(0,h_{2i}+(t+1)h_{2i+1}-1\big), 0\leq t< a_{2i+2}.}
	The corresponding translated rectangles (translation applied to each point) denoted by $\rectangle \ti{P^t_i}\ti{Q^t_i}\ti{R^t_i}\ti{S^t_i}$ of next numbers in the multiplicatively closed set are given by 
	\equa{	\ti{P^t_i}&=\big(k_{2i}+tk_{2i+1},0\big),\\
			\ti{Q^t_i}&=\big(k_{2i}+(t+1)k_{2i+1}-1,0\big)\\
			\ti{R^t_i}&=\big(k_{2i}+(t+1)k_{2i+1}-1,h_{2i+1}-1\big),\\
			\ti{S^t_i}&=\big(k_{2i}+tk_{2i+1},h_{2i+1}-1\big), 0\leq t< a_{2i+2}}
	again of the same dimensions $h_{2i+1}\times k_{2i+1}$
	with translation given by \equ{\rectangle\ti{P^t_i}\ti{Q^t_i}\ti{R^t_i}\ti{S^t_i}=\rectangle P^t_iQ^t_iR^t_iS^t_i+\big(k_{2i}+tk_{2i+1},-h_{2i}-th_{2i+1}\big).}
	
	Also we have the grid
	\equa{\big(\mbb{N}\cup\{0\}\big)\times \big(\mbb{N}\cup\{0\}\big) &=\\
&\us{i\geq 1}{\bigsqcup}\bigg( \us{0\leq t< a_{2i+1}}{\bigsqcup} \rectangle A^t_iB^t_iC^t_iD^t_i \bigg)  
		\us{i\geq 0}{\bigsqcup}\bigg( \us{0\leq t< a_{2i+2}}{\bigsqcup} \rectangle P^t_iQ^t_iR^t_iS^t_i \bigg)}
	and the grid of next numbers (hence origin deleted as next number cannot be origin)
	\equa{\big(\mbb{N}\cup\{0\}\big)&\times \big(\mbb{N}\cup\{0\}\big)\bs\{(0,0)\} =\\
		&\us{i\geq 1}{\bigsqcup }\bigg( \us{0\leq t< a_{2i+1}}{\bigsqcup} \rectangle \ti{A^t_i}\ti{B^t_i}\ti{C^t_i}\ti{D^t_i} \bigg) 
		\us{i\geq 0}{\bigsqcup }\bigg( \us{0\leq t< a_{2i+2}}{\bigsqcup} \rectangle \ti{P^t_i}\ti{Q^t_i}\ti{R^t_i}\ti{S^t_i} \bigg).}
\end{thmOmega}
Here in the following note we mention briefly how to get the factorization of the previous number of $s_k\in S$.
\begin{note}
	\label{note:AdjFact}
	Using Theorem~\ref{theorem:FactRectangles} the factorization of the previous number can be obtained because previous number of the next number of a number is the given number. The answer to Question~\ref{ques:AdjacentNumberFactorization} can be obtained by suitably expressing $(i,j)$. 
	
	For the next number express
	$s_k=(i,j)\in (\mbb{N}\cup \{0\})\times (\mbb{N}\cup \{0\})$ as $(k_{2l-1}+tk_{2l}+r,s)$ with $0\leq t<a_{2l+1},0\leq r<k_{2l},0\leq s<h_{2l}$ or as 
	$(r,h_{2l}+th_{2l+1}+s)$ with $0\leq t<a_{2l+2},0\leq r<k_{2l+1},0\leq s<h_{2l+1}$. We get the next number $s_{k+1}$ using Theorem~\ref{theorem:FactRectangles}.
	
	For the previous number express $s_{k}=(i,j)\in (\mbb{N}\cup \{0\})\times (\mbb{N}\cup \{0\})\bs\{0,0\}$ as $(r,h_{2l-1}+th_{2l}+s)$ with $0\leq t<a_{2l+1},0\leq r<k_{2l},0\leq s<h_{2l}$
	or as $(r+k_{2l}+tk_{2l+1},s)$ with $0\leq t<a_{2l+2},0\leq r<k_{2l+1},0\leq s<h_{2l+1}$. We get the previous number $s_{k-1}$ again using Theorem~\ref{theorem:FactRectangles} in the reverse manner.
\end{note}
\section{\bf{Types of Fractions and Convergents associated to an Irrational in $[0,1]$}}
In this section we define various types of fractions and convergents associated to an irrational $\ga \in [0,1]$.

\subsection{Primary and Secondary Convergents}
\begin{defn}[Primary and Secondary Convergents]
Let $\ga \in [0,1]$ be an irrational. Suppose $\{a_0=0,a_1,a_2,a_3,\cdots,\}$ be the sequence denoting the simple continued fraction expansion of $\ga$, i.e.,
\equ{\ga=a_0+\frac 1{a_1+\frac 1{a_2+\frac 1{a_3+ \frac 1{\cdots}}}}.}
Let $h_0=0,k_0=1$. Define \equ{\frac{h_i}{k_i}=a_0+\frac 1{a_1+\frac 1{\ddots+\frac 1{a_i}}}, gcd(h_i,k_i)=1 \text{ for } i\in \mbb{N}.}
Then an element in the sequence $\{\frac{h_i}{k_i}: i\in \mbb{N}\cup\{0\}\}$ is called a primary convergent. 
The first few primary convergents with relatively prime numerators and denominators are given by 
\equ{\frac 01,\frac 1{a_1},\frac{a_2}{1+a_1a_2},\frac{1+a_2a_3}{a_3+a_1+a_1a_2a_3},\frac{a_2+a_4+a_2a_3a_4}{1+a_3a_4+a_1a_2+a_1a_4+a_1a_2a_3a_4},\cdots.}
By induction, we can show, with these expressions for $\frac{h_i}{k_i}$, that $h_ik_{i+1}-k_ih_{i+1}=\pm 1,i\in \mbb{N}\cup\{0\}$ as polynomials. Also we have as polynomials, 
\equ{h_{i+2}=a_{i+2}h_{i+1}+h_i, k_{i+2}=a_{i+2}k_{i+1}+k_i.}
So we actually have polynomial expressions for $h_i,k_i$ in terms of variables $a_i:i\in \mbb{N}\cup\{0\}$ arising from continued fraction of the irrational $\ga$. For any irrational $\ga$
the convergents satisfy \equ{\frac{h_{2j}}{k_{2j}}<\frac{h_{2j+2}}{k_{2j+2}}<\ga < \frac{h_{2l+1}}{k_{2l+1}}<\frac{h_{2l-1}}{k_{2l-1}}, j\in \mbb{N}\cup \{0\}, l\in \mbb{N}.}

Now we define the finite monotonic sequences of new intermediate fractions with relatively prime numerators and denominators given by 
\equ{\frac{h_{2j}}{k_{2j}}<\frac{h_{2j}+th_{2j+1}}{k_{2j}+tk_{2j+1}}<\frac{h_{2j+2}}{k_{2j+2}}, 0<t<a_{2j+2}, t,j\in \mbb{N}\cup\{0\}}
and 
\equ{\frac{h_{2l+1}}{k_{2l+1}}<\frac{h_{2l-1}+th_{2l}}{k_{2l-1}+tk_{2l}}<\frac{h_{2l-1}}{k_{2l-1}}, 0<t<a_{2l+1}, t,l\in \mbb{N}.}
These new intermediate fractions are called secondary convergents.
\end{defn}

\subsection{Upper and Lower Fractions}
Here we define two sequences of fractions called upper and lower fractions associated to an irrational $\ga\in [0,1]$.

\begin{defn}[Upper and and Lower Fractions]
Let $0<\ga<1$. For $n\in \mbb{N}$, let \equ{f(n)=\bigg\lceil \frac n{\ga} \bigg\rceil, g(n)=\bigg\lfloor \frac n{\ga} \bigg\rfloor.} 
Then the sequences $\{f(n):n\in \mbb{N}\},\{g(n):n\in \mbb{N}\}$ are called upper and lower
sequences of $\ga$ respectively. We have \equ{g(n)\ga < n < f(n)\ga, f(n)-g(n)=1, n\in \mbb{N}.} Since $0<\ga<1$ for any $n\in \mbb{N}$, we also observe that 
\equ{\lfloor f(n)\ga \rfloor = n, \lceil g(n)\ga \rceil = n.} 
A fraction in the sequence $\{\frac n{f(n)}:n\in \mbb{N}\}$ is called a lower fraction associated to $\ga$ and a fraction in the sequence $\{\frac n{g(n)}:n\in \mbb{N}\}$ is called an upper fraction
associated to $\ga$. We need not in general have $gcd(n,f(n))=1$ or $gcd(n,g(n))=1$. However we definitely have \equ{\frac {n}{f(n)}<\ga < \frac{n}{g(n)}.}
\end{defn}
\begin{note}
An element in the upper sequence gives a lower fraction and an element in the lower sequence gives an upper fraction associated to $\ga$.
\end{note}
\subsection{The Upper and Lower Sequences $f$ and $g$}
Here in this section we prove Theorems~\ref{theorem:UpLowSequence}~\ref{theorem:UpLowSeqDifferences} concerning the values of the upper and lower sequences. 
\begin{theorem}
\label{theorem:UpLowSequence}
Let $\ga \in [0,1]$ be an irrational with continued fraction expansion $\{a_0=0,a_1,a_2,\cdots\}$. 
Let $\{f(n):n\in \mbb{N}\},\{g(n):n\in \mbb{N}\}$ be the upper and lower sequences of $\ga$.
Let $h_0=0,k_0=1$. For $i\in \mbb{N}$, let $h_i,k_i$ be the numerator and denominator of $i^{th}$ primary convergent which are relatively prime.
Then we have 
\equa{f(h_{2j}+th_{2j+1})&=k_{2j}+tk_{2j+1},\\
g(h_{2j}+th_{2j+1})&=k_{2j}+tk_{2j+1}-1, 0<t\leq a_{2j+2},j\in \mbb{N}\cup\{0\},t\in \mbb{N}\\
g(h_{2l-1}+th_{2l})&=k_{2l-1}+tk_{2l},\\
f(h_{2l-1}+th_{2l})&=k_{2l-1}+tk_{2l}+1,0 \leq t \leq a_{2l+1},t\in \mbb{N}\cup\{0\},l\in \mbb{N}.}  
\end{theorem}
\begin{proof}
To prove the theorem it suffices to prove the following inequalities. 
\equa{(k_{2j}+tk_{2j+1}-1)\ga<h_{2j}+th_{2j+1}&<(k_{2j}+tk_{2j+1})\ga,\\
	&0<t\leq a_{2j+2},j\in \mbb{N}\cup\{0\},t\in \mbb{N}\\
(k_{2l-1}+tk_{2l})\ga<h_{2l-1}+th_{2l}&<(k_{2l-1}+tk_{2l}+1)\ga,\\
&0\leq t \leq a_{2l+1},t,l\in \mbb{N}.}
This we prove by induction on $j,l$ simultaneously as follows. 

We observe that 
\equ{h_0=k_0-1=0,k_1\ga < h_1 \Ra (k_0+tk_1-1)\ga<h_0+th_1\text{ for all }t > 0.}
We also have for 
\equ{0 \leq t \leq a_2, \frac{h_0+th_1}{k_0+tk_1}\leq \frac{h_2}{k_2}<\ga \Ra h_0+th_1 < (k_0+tk_1)\ga.} So
\equ{(k_0+tk_1-1)\ga<h_0+th_1<(k_0+tk_1)\ga, 0<t\leq a_2.}
We have 
\equ{h_2<k_2\ga\text{ and }\frac 1{a_1+1}<\ga \Ra h_1 < (k_1+1)\ga.}
Hence \equ{\text{for all }t \geq 0, h_1+th_2<(k_1+tk_2+1)\ga.}
We also have 
\equ{\text{for }0\leq t \leq a_3, \ga<\frac{h_3}{k_3}\leq \frac{h_1+th_2}{k_1+tk_2} \Ra (k_1+tk_2)\ga < h_1+th_2.} 
So
\equ{(k_1+tk_2)\ga < h_1+th_2 < (k_1+tk_2+1)\ga, 0 \leq t \leq a_3.}
This proves the initial step of the induction for $j=0,l=1$. 

Now assume that the inequalities follow for $j=r,l=r+1$ for some $r\in \mbb{N}$. We prove for $j=r+1,l=r+2$.
We have 
\equ{k_{2r+3}\ga < h_{2r+3}\text{ and for }j=r,t=a_{2r+2},(k_{2r+2}-1)\ga < h_{2r+2}} 
which together imply 
\equ{(k_{2r+2}+tk_{2r+3}-1)\ga < h_{2r+2}+th_{2r+3}\text{ for }t\geq 0.}
We also have for 
\equa{0\leq t\leq a_{2r+4}, \frac{h_{2r+2}+th_{2r+3}}{k_{2r+2}+tk_{2r+3}} &\leq \frac{h_{2r+4}}{k_{2r+4}}<\ga \Ra\\
h_{2r+2}+th_{2r+3}&<(k_{2r+2}+tk_{2r+3})\ga.}
So
\equ{(k_{2r+2}+tk_{2r+3}-1)\ga < h_{2r+2}+th_{2r+3}<(k_{2r+2}+tk_{2r+3})\ga.}
We have \equ{h_{2r+4}<k_{2r+4}\ga\text{ and for }l=r+1,t=a_{2r+3}, h_{2r+3}<(k_{2r+3}+1)\ga} 
which together impy
\equ{h_{2r+3}+th_{2r+4}<(k_{2r+3}+tk_{2r+4}+1)\ga\text{ for all }t\geq 0.}
We also have for 
\equa{0\leq t\leq a_{2r+5}, \ga < \frac{h_{2r+5}}{k_{2r+5}} &\leq \frac{h_{2r+3}+th_{2r+4}}{k_{2r+3}+tk_{2r+4}} \Ra \\
	(k_{2r+3}+tk_{2r+4})\ga &< h_{2r+3}+th_{2r+4}.}
So
\equ{(k_{2r+3}+tk_{2r+4})\ga<h_{2r+3}+th_{2r+4}<(k_{2r+3}+tk_{2r+4}+1)\ga, 0\leq t\leq a_{2r+5}.}
This proves the induction step for $j=r+1,l=r+2$. 

Hence the theorem follows and the values of the sequences $f,g$ at the values of $n$ being the numerator of any primary or secondary convergent are known. 
\end{proof}
\begin{note}
\label{note:FracPart}
Now we make an important observation about the monotonic nature of the ceil or rounding up fractional parts $h_{*}-k_{*}\ga$.
The numerators of the secondary and primary convergents associated to the lower sequence $g$ satisfy the following monotonicity.
\equ{h_1<h_1+h_2<\cdots<h_1+a_3h_2=h_3<h_3+h_4<\cdots<h_3+a_5h_4=h_5<\cdots.}
The sequence of differences is given by \equ{h_2,h_2,\cdots,h_2,h_4,h_4,\cdots,h_4,\cdots}
where $h_{2i}$ appears $a_{2i+1}$ times for $i\geq 1$. This sequence is non-decreasing and diverges to infinity.
Now we apply $g$ to the above sequence to obtain the denominators of the secondary and primary convergents associated to the lower sequence $g$ which also satisfy the following monotonicity.
\equ{k_1<k_1+k_2<\cdots<k_1+a_3k_2=k_3<k_3+k_4<\cdots<k_3+a_5k_4=k_5<\cdots.}
The sequence of differences is given by \equ{k_2,k_2,\cdots,k_2,k_4,k_4,\cdots,k_4,\cdots}
where $k_{2i}$ appears $a_{2i+1}$ times for $i\geq 1$. This sequence is non-decreasing and diverges to infinity.
The ceil fractional parts satisfy
\equa{h_1-k_1\ga>(h_1+h_2)&-(k_1+k_2)\ga>\cdots>h_3-k_3\ga>\\
	&(h_3+h_4)-(k_3+k_4)\ga>\cdots>h_5-k_5\ga>\cdots.}

Similarly we make an observation on the monotonic nature of the floor or usual fractional parts $k_{*}\ga-h_{*}$.
The numerators of the secondary and primary convergents associated to the upper sequence $f$ satisfy the following monotonicity.
\equ{h_0<h_0+h_1<\cdots<h_0+a_2h_1=h_2<h_2+h_3<\cdots<h_2+a_4h_3=h_4<\cdots.}
The sequence of differences  is given by \equ{h_1,h_1,\cdots,h_1,h_3,h_3,\cdots,h_3,\cdots}
where $h_{2i-1}$ appears $a_{2i}$ times for $i\geq 1$. This sequence is non-decreasing and diverges to infinity.
Now we apply $f$ to the above sequence to obtain the denominators of the secondary and primary convergents associated to the upper sequence $f$ which also satisfy the following monotonicity.
\equ{k_0+k_1<\cdots<k_0+a_2k_1=k_2<k_2+k_3<\cdots<k_2+a_4k_3=k_4<\cdots.}
The sequence of differences after including $k_0$ in the beginning is given by 
\equ{k_1,k_1,\cdots,k_1,k_3,k_3,\cdots,k_3,\cdots}
where $k_{2i-1}$ appears $a_{2i}$ times for $i\geq 1$. This sequence is non-decreasing and diverges to infinity.
The floor fractional parts satisfy
\equa{k_0\ga>(k_0+k_1)\ga&-(h_0+h_1)>\cdots>k_2\ga-h_2>\\
	&(k_2+k_3)\ga-(h_2+h_3)>\cdots>k_4\ga-h_4>\cdots.}
\end{note}
Now we prove a useful lemma regarding fractions.
\begin{lemma}
\label{lemma:fracDetOne}
Let $\frac ab>\frac pq>\frac cd \geq 0$ be three fractions such that $ad-bc=1$. Then we have \equ{q>max(b,d).}
\end{lemma}
\begin{proof}
We have $\frac 1{bd}>\frac ab-\frac pq=\frac {aq-bp}{bq}$. If $aq-bp=1$ then $q>d$. If $aq-bp>1$ and $q\leq d$ then $\frac {aq-bp}{bq}\geq \frac {aq-bp}{bd}>\frac 1{bd}$ a contradiction.
Hence $q>d$.
We also have $\frac 1{bd}>\frac pq-\frac cd=\frac {pd-cq}{dq}$. If $pd-cq=1$ then $q>b$. If $pd-cq>1$ and $q\leq b$ then $\frac {pd-cq}{dq}\geq \frac {pd-cq}{bd}>\frac 1{bd}$ a contradiction.
Hence $q>b$. 
So the lemma follows.
\end{proof}
We prove the second theorem.
\begin{theorem}
\label{theorem:UpLowSeqDifferences}
Let $\ga \in [0,1]$ be an irrational. 
Let $\{f(n):n\in \mbb{N}\},\{g(n):n\in \mbb{N}\}$ be the upper and lower sequences of $\ga$. 
Consider the two sequences of lower and upper fractions of primary and secondary convergents respectively.
\equa{\bigg\{\frac {p_1}{q_1}&<\cdots<\frac {p_i}{q_i}<\cdots\mid i\in \mbb{N}\bigg\}=\\
&\bigg\{\frac {h_0}{k_0}<\frac {h_0+h_1}{k_0+k_1}<\cdots <\frac{h_2}{k_2}<\frac {h_2+h_3}{k_2+k_3}<\cdots <\frac{h_4}{k_4}<\frac {h_4+h_5}{k_4+k_5}<\cdots\bigg\}.}
\equa{\bigg\{\frac {r_1}{s_1}&>\cdots>\frac {r_i}{s_i}>\cdots\mid i\in \mbb{N}\bigg\}=\\
&\bigg\{\frac {h_1}{k_1}>\frac {h_1+h_2}{k_1+k_2}>\cdots >\frac{h_3}{k_3}>\frac {h_3+h_4}{k_3+k_4}>\cdots >\frac{h_5}{k_5}>\frac {h_5+h_6}{k_5+k_6}>\cdots\bigg\}.}
Given a lower fraction $\frac{n}{f(n)}\nin \{\frac {p_j}{q_j}\mid j\in \mbb{N}\}$ there exists a lower fraction $\frac {p_i}{q_i}$ such that 
\equ{n>p_i, f(n)>f(p_i)=q_i,f(n)\ga-n>q_i\ga-p_i.}
Given an upper fraction $\frac{n}{g(n)}\nin \{\frac {r_j}{s_j}\mid j\in \mbb{N}\}$ there exists an upper fraction $\frac {r_i}{s_i}$ such that 
\equ{n>r_i,g(n)>g(r_i)=s_i,n-g(n)\ga>r_i-s_i\ga.}
\end{theorem}
\begin{proof}
Both the sequences $f,g$ are monotonically increasing and $\frac{h_i}{k_i}-\frac{h_{i+1}}{k_{i+1}}=\frac{(-1)^{i+1}}{k_ik_{i+1}}\lra 0$ as $i\lra \infty$.
Hence \equ{\us{i\lra \infty}{\lim} \frac{p_i}{q_i}=\us{i\lra \infty}{\lim} \frac{h_i}{k_i}= \us{i\lra \infty}{\lim} \frac{r_i}{s_i}=\ga.} 
Now $0<\frac n{f(n)} <\ga$. So there exist two consecutive lower fractions $\frac{p_{i-1}}{q_{i-1}},\frac{p_i}{q_i}$ such that \equ{\frac{p_{i-1}}{q_{i-1}}<\frac n{f(n)} <\frac {p_i}{q_i}<\ga.}
Using Lemma~\ref{lemma:fracDetOne} we conclude that $f(n)>max(q_i,q_{i-1})=q_i>q_{i-1}$ since $p_iq_{i-1}-q_ip_{i-1}$ $=1$. 
By monotonicity of $f$ we conclude that $n> max(p_i,p_{i-1})=p_i>p_{i-1}$.
Now we have \equ{f(n)\ga-n=f(n)(\ga-\frac n{f(n)})>f(n)(\ga-\frac {p_i}{q_i})>q_i(\ga-\frac {p_i}{q_i})=q_i\ga-p_i.}
Now we prove the other case.  Suppose  $\frac n{g(n)}>\frac {h_1}{k_1}=\frac 1{k_1}>\ga$. Since $n>1,g(n)>k_1$. Choose $\frac{r_i}{s_i}=\frac {h_1}{k_1}$ and we have 
\equ{n-g(n)\ga=g(n)(\frac n{g(n)}-\ga)>g(n)(\frac {h_1}{k_1}-\ga)>k_1(\frac {h_1}{k_1}-\ga)=h_1-k_1\ga.}
Suppose $\frac{h_1}{k_1}> \frac n{g(n)} > \ga$. Then there exist two consecutive upper fractions $\frac {r_{i-1}}{s_{i-1}},\frac {r_i}{s_i}$ such that 
\equ{\frac {r_{i-1}}{s_{i-1}}>\frac n{g(n)}>\frac {r_i}{s_i}>\ga.}
Again using Lemma~\ref{lemma:fracDetOne} we conclude that $g(n)>max(s_i,s_{i-1})=s_i>s_{i-1}$ since $s_ir_{i-1}-r_is_{i-1}=1$.
By monotonicity of $g$ we conclude that $n>max(r_i,r_{i-1})=r_i>r_{i-1}$.
Now we have \equ{n-g(n)\ga=g(n)(\frac n{g(n)}-\ga)>g(n)(\frac {r_i}{s_i}-\ga)>s_i(\frac {r_i}{s_i}-\ga)=r_i-s_i\ga.}
This proves the theorem.   
\end{proof}
\begin{note}
\label{note:MinFracPart}
In Theorem~\ref{theorem:UpLowSeqDifferences} if $\frac{n}{f(n)}=\frac{p_i}{q_i}$ then there exists a positive integer $k$ such that $n=kp_i,f(n)=kq_i$. 
So if $k>1$ the we have $n>p_i,f(n)>q_i,f(n)\ga-n>q_i\ga-p_i$. The other case is similar. So we conclude that minimal fractional parts occur exactly at the numerators
and denominators of primary and secondary convergents using Note~\ref{note:FracPart} in the sequences of lower and upper fractions of $\ga$ which is the content of the statement of Theorem~\ref{theorem:FracPart}.
\end{note}
\section{\bf{The proof of the main theorem}}
\label{sec:FactRectangles}
We begin this section with a theorem which is an observation on the rounding up or ceil fractional parts of the lower sequence and the usual fractional parts of the upper sequence.

\begin{theorem}
\label{theorem:FracPart}
Let $\ga\in [0,1]$ be an irrational. Let $f,g$ be the upper and lower sequences associated to $\ga$. Let $z_0=\ga$ and for $n\in \mbb{N}$ let $z_n=-n+f(n)\ga,y_n=n-g(n)\ga$.
Let $n_0=0,z_{n_0}=z_0=\ga,m_1=1,y_{m_1}=y_1=1-g(1)\ga$. Define two subsequences $z_{n_j},y_{m_j}$
with the property that 
\equa{z_{n_j}<z_{n_{j-1}}&=\min \{z_0,\cdots, z_{n_j-1}\} \text{ for }j\in \mbb{N},\\
y_{m_j}<y_{m_{j-1}}&=\min \{y_1,\cdots, y_{m_j-1}\} \text{ for }1<j\in \mbb{N}.}
Then the sequence \equa{\{n_0<n_1&<\cdots<n_j<\cdots\mid j\in \mbb{N}\cup\{0\}\}=\\
&\{h_0<h_0+h_1<\cdots<h_0+a_2h_1=h_2<h_2+h_3<\cdots<\\
&h_2+a_4h_3=h_4<h_4+h_5<\cdots\}} and the sequence
\equa{\{m_1<m_2&<\cdots<m_j<\cdots\mid j\in \mbb{N}\}=\\
&\{h_1<h_1+h_2<\cdots<h_1+a_3h_2=h_3<h_3+h_4<\cdots<\\
&h_3+a_5h_4=h_5<h_5+h_6<\cdots\}.}
\end{theorem}
\begin{proof}
This theorem follows by applying Theorem~\ref{theorem:UpLowSeqDifferences} and Notes~\ref{note:FracPart},~\ref{note:MinFracPart} which together imply
that the lesser fractional parts in the sequence occur exactly at numerators of primary and secondary convergents of the lower and upper fractions for
the upper and lower sequences respectively.
\end{proof}
Now we prove main Theorem~\ref{theorem:FactRectangles}.
\begin{proof}
Let $\ga=\frac {\log p_1}{\log p_2}\in \mbb{R}\bs \mbb{Q}$.
Consider a point \equ{(k_{2i-1}+tk_{2i}+r,s)\in \rectangle A^t_iB^t_iC^t_iD^t_i} and its next number \equ{(r,h_{2i-1}+th_{2i}+s)\in \rectangle \ti{A^t_i}\ti{B^t_i}\ti{C^t_i}\ti{D^t_i}}
for any \equ{0\leq t<a_{2i+1},0\leq r <k_{2i},0\leq s< h_{2i}.} Now suppose there exists an integer $p_1^bp_2^a$ such that 
\equ{p_1^{k_{2i-1}+tk_{2i}+r}p_2^{s}<p_1^bp_2^a<p_1^{r}p_2^{h_{2i-1}+th_{2i}+s}.}
Then we arrive at a contradiction as follows.

The following sequence of inequalities hold 
\equ{\big(k_{2i-1}+tk_{2i}+r\big)\ga+s<b\ga+a<r\ga+h_{2i-1}+th_{2i}+s<\big(k_{2i-1}+tk_{2i}+r+1\big)\ga+s}
because $\ga=\frac{\log p_1}{\log p_2}$.
If $b=r$ then there exist two integers $a-s,h_{2i-1}+th_{2i}$ in between \equ{\big(k_{2i-1}+tk_{2i}\big)\ga,\big(k_{2i-1}+tk_{2i}+1\big)\ga} which is a contradiction.

If $b<r$ then we must have $a>s$ so that $a-s>0$. Hence
\equ{\big(k_{2i-1}+tk_{2i}+r-b\big)\ga<a-s<h_{2i-1}+th_{2i}+(r-b)\ga<\big(k_{2i-1}+tk_{2i}+r-b+1\big)\ga.}
So let $0<z<\ga$ be the positive rounding up ceil fractional part such that 
\equ{\big(k_{2i-1}+tk_{2i}+r-b\big)\ga+z=a-s.}
Now suppose $k_{2i-1}+tk_{2i}+r-b<k_{2i-1}+(t+1)k_{2i}$. Then the fractional part $z$ has to satisfy 
\equ{z\geq h_{2i-1}+th_{2i}-\big(k_{2i-1}+tk_{2i}\big)\ga}
because minimal fractional parts occur exactly at the numerators and denominators of primary and secondary convergents.
On the other hand \equa{z=(a-s)&-\big(k_{2i-1}+tk_{2i}+r-b\big)\ga<\\
	&h_{2i-1}+th_{2i}+(r-b)\ga-\big(k_{2i-1}+tk_{2i}+r-b\big)\ga=\\
	&h_{2i-1}+th_{2i}-\big(k_{2i-1}+tk_{2i}\big)\ga}
which is a contradiction. 
So we have \equ{k_{2i-1}+tk_{2i}+r-b\geq k_{2i-1}+(t+1)k_{2i}\Ra r-b \geq k_{2i} \Ra r\geq k_{2i}+b\Ra r\geq k_{2i}}
which is again a contradiction to $0\leq r\leq k_{2i}-1$.

Now consider the case $b>r$ so that $b-r>0$. We have 
\equ{\big(k_{2i-1}+tk_{2i}\big)\ga+s-a<(b-r)\ga<h_{2i-1}+th_{2i}+s-a<\big(k_{2i-1}+tk_{2i}+1)\ga+s-a.}

Let $0<y<\ga$ be the rounding up ceil fractional part such that 
\equ{(b-r)\ga+y=h_{2i-1}+th_{2i}+s-a.}
Now suppose $h_{2i-1}+th_{2i}+s-a<h_{2i-1}+(t+1)h_{2i}$. Then the fractional part $y$ has to satisfy 
\equ{y\geq h_{2i-1}+th_{2i}-\big(k_{2i-1}+tk_{2i}\big)\ga}
because minimal fractional parts occur exactly at the numerators and denominators of primary and secondary convergents.
On the other hand \equa{y&=h_{2i-1}+th_{2i}+s-a-(b-r)\ga<\\
	&h_{2i-1}+th_{2i}+s-a-\big((k_{2i-1}+tk_{2i})\ga+s-a\big)=\\
&h_{2i-1}+th_{2i}-\big(k_{2i-1}+tk_{2i}\big)\ga}
which is a contradiction. So we have 
\equ{h_{2i-1}+th_{2i}+s-a \geq h_{2i-1}+(t+1)h_{2i} \Ra s-a\geq h_{2i} \Ra s\geq h_{2i}+a \Ra s\geq h_{2i}}
which is again a contradiction to $0\leq s\leq h_{2i}-1$.

This proves that the next number of $p_1^{k_{2i-1}+tk_{2i}+r}p_2^{s}$ is $p_1^{r}p_2^{h_{2i-1}+th_{2i}+s}$ for 
\equ{0\leq t<a_{2i+1},0\leq r <k_{2i},0\leq s<h_{2i}.} 

Now we consider the second set of rectangles. Consider a point 
\equ{(r,h_{2i}+th_{2i+1}+s)\in \rectangle P^t_iQ^t_iR^t_iS^t_i} and its next number \equ{(r+k_{2i}+tk_{2i+1},s)\in \rectangle \ti{P^t_i}\ti{Q^t_i}\ti{R^t_i}\ti{S^t_i}}
for any \equ{0\leq t<a_{2i+2},0\leq r<k_{2i+1},0\leq s< h_{2i+1}.} Now suppose there exists an integer $p_1^bp_2^a$ such that 
\equ{p_1^{r}p_2^{h_{2i}+th_{2i+1}+s}<p_1^bp_2^a<p_1^{r+k_{2i}+tk_{2i+1}}p_2^{s}.}
Then we arrive at a contradiction as follows in a similar manner.

The following sequence of inequalities hold  
\equ{\big(k_{2i}+tk_{2i+1}+r-1\big)\ga+s<r\ga+h_{2i}+th_{2i+1}+s<b\ga+a<\big(k_{2i}+tk_{2i+1}+r\big)\ga+s}
because $\ga=\frac{\log p_1}{\log p_2}$. This implies
\equ{\big(k_{2i}+tk_{2i+1}+r-b-1\big)\ga<h_{2i}+th_{2i+1}+(r-b)\ga<a-s<\big(k_{2i}+tk_{2i+1}+r-b\big)\ga.}
If $a=s$ then we observe that $\big(k_{2i}+tk_{2i+1}+r-b\big)\ga\geq \ga$ and $\big(k_{2i}+tk_{2i+1}+r-b-1\big)\ga\leq -\ga$ which is a contradiction.
If $a>s$ so that $a-s>0$ then let $0<x<\ga$ be the positive fractional part such that 
\equ{(a-s)+x=\big(k_{2i}+tk_{2i+1}+r-b\big)\ga.}
Now suppose $\big(k_{2i}+tk_{2i+1}+r-b\big)<k_{2i}+(t+1)k_{2i+1}$. Then the fractional part $x$ has to satisfy
\equ{x\geq (k_{2i}+tk_{2i+1})\ga-(h_{2i}+th_{2i+1})}
because minimal fractional parts occur exactly at the numerators and denominators of primary and secondary convergents.
On the other hand
\equa{x&=\big(k_{2i}+tk_{2i+1}+r-b\big)\ga-(a-s)<\\
	&\big(k_{2i}+tk_{2i+1}+r-b\big)\ga-\big(h_{2i}+th_{2i+1}+(r-b)\ga\big)=\\
&(k_{2i}+tk_{2i+1})\ga-(h_{2i}+th_{2i+1})}
which is a contradiction. So we have
\equa{\big(k_{2i}+tk_{2i+1}+r-b\big)&\geq k_{2i}+(t+1)k_{2i+1} \Ra\\
	 r-b&\geq k_{2i+1}\Ra r\geq k_{2i+1}+b\Ra r\geq k_{2i+1}}
which is a contradiction to $0\leq r\leq k_{2i+1}-1$.

Now suppose $a<s$ so that $s-a>0$. Then we have 
\equ{\big(k_{2i}+tk_{2i+1}-1\big)\ga+s-a<h_{2i}+th_{2i+1}+s-a<(b-r)\ga<\big(k_{2i}+tk_{2i+1}\big)\ga+s-a.}
So $b-r>0$ and let $0<u<\ga$ be the positive fractional part such that
\equ{h_{2i}+th_{2i+1}+s-a+u=(b-r)\ga.}
Now suppose $h_{2i}+th_{2i+1}+s-a<h_{2i}+(t+1)h_{2i+1}$. Then the fractional part $u$ has to satisfy
\equ{u\geq (k_{2i}+tk_{2i+1})\ga-(h_{2i}+th_{2i+1})}
because minimal fractional parts occur exactly at the numerators and denominators of primary and secondary convergents.
On the other hand
\equa{	u&=(b-r)\ga-(h_{2i}+th_{2i+1}+s-a)<\\
		&\big(k_{2i}+tk_{2i+1}\big)\ga+s-a-(h_{2i}+th_{2i+1}+s-a)=\\
		&(k_{2i}+tk_{2i+1})\ga-(h_{2i}+th_{2i+1})}
which is a contradiction. So we have
\equa{h_{2i}+th_{2i+1}+s-a&>h_{2i}+(t+1)h_{2i+1}\Ra\\ 
	s-a&\geq h_{2i+1}\Ra s\geq h_{2i+1}+a \Ra s\geq h_{2i+1}}
which is again a contradiction to $0\leq s\leq h_{2i+1}-1$.

This proves that the next number of $p_1^{r}p_2^{h_{2i}+th_{2i+1}+s}$ is $p_1^{r+k_{2i}+tk_{2i+1}}p_2^{s}$ for 
\equ{0\leq t<a_{2i+2},0\leq r <k_{2i+1},0\leq s<h_{2i+1}.} 

Now we prove that the rectangles cover the grid $(\mbb{N}\cup\{0\})\times(\mbb{N}\cup\{0\})$. For this we need to prove that given 
$(x,y)\in (\mbb{N}\cup\{0\})\times(\mbb{N}\cup\{0\})$ either there exists $j \geq 0$ such that 
\equ{h_{2j}+th_{2j+1} \leq y < h_{2j}+(t+1)h_{2j+1}, 0 \leq x < k_{2j+1}\text{ for some }0\leq t<a_{2j+2}} 
or there exists $i\geq 1$ such that 
\equ{k_{2i-1}+\ti{t}k_{2i}\leq x <k_{2i-1}+(\ti{t}+1)k_{2i},0\leq y<h_{2i} \text{ for some }0\leq \ti{t}<a_{2i+1}} 

Now there always exist $j\geq 0,0\leq t<a_{2j+2}$ such that \equ{h_{2j}+th_{2j+1} \leq y < h_{2j}+(t+1)h_{2j+1}.}
If $0 \leq x < k_{2j+1}$ then we are done. Otherwise $x\geq k_{2j+1}$. Hence there exist $i\geq j+1, 0\leq \ti{t}<a_{2i+1}$ such that 
\equ{k_{2i-1}+\ti{t}k_{2i}\leq x <k_{2i-1}+(\ti{t}+1)k_{2i}.} Now we have \equ{0\leq h_{2j}+th_{2j+1} \leq y < h_{2j}+(t+1)h_{2j+1}\leq h_{2j+2}\leq h_{2i} \Ra 0\leq y<h_{2i}.}
So the rectangles cover the grid. The rest of the proof for the grid with origin deleted is similar. 

This completes the proof of Theorem~\ref{theorem:FactRectangles}.
\end{proof}
We mention the following corollary of Theorem~\ref{theorem:FactRectangles} with a very brief proof as it is straight forward.
\begin{cor}
\label{cor:ArbLargeInt}
Let $S$ be a multiplicatively closed set with two generators $1<p_1<p_2$ such that $\frac{\log(p_1)}{\log(p_2)}$ is irrational.  Then there exist arbitrarily large
gap integer intervals of $S$ with end points in $S$.
\end{cor}
\begin{proof}
To obtain arbitrarily large gap intervals, we apply Theorem~\ref{theorem:FactRectangles} for the largest values of $r\in \{k_{2i}-1,k_{2i+1}-1\}, s\in\{h_{2i}-1,h_{2i+1}-1\}$ 
which tend to infinity as $i$ tends to infinity.
\end{proof}

\end{document}